\documentclass[preprint,12pt]{elsarticle}

\usepackage{amssymb}
\usepackage{amsthm}
\usepackage{amsmath}
\usepackage{amsfonts}

\newtheorem{theorem}{Theorem}
\newtheorem{lemma}{Lemma}
\newtheorem{corollary}{Corollary}
\usepackage{ragged2e}

\usepackage{newtxtext,newtxmath}

\usepackage[utf8]{inputenc}
\journal{arXiv }

\begin{document}
%\fontfamily{times}
\begin{frontmatter}

%% Title, authors and addresses

%% use the tnoteref command within \title for footnotes;
%% use the tnotetext command for the associated footnote;
%% use the fnref command within \author or \address for footnotes;
%% use the fntext command for the associated footnote;
%% use the corref command within \author for corresponding author footnotes;
%% use the cortext command for the associated footnote;
%% use the ead command for the email address,
%% and the form \ead[url] for the home page:
%%
%% \title{Title\tnoteref{label1}}
%% \tnotetext[label1]{}
%% \author{Name\corref{cor1}\fnref{label2}}
%% \ead{email address}
%% \ead[url]{home page}
%% \fntext[label2]{}
%% \cortext[cor1]{}
%% \address{Address\fnref{label3}}
%% \fntext[label3]{}

\title{Solutions to Diophantine equation of Erdős–Straus Conjecture}

%% use optional labels to link authors explicitly to addresses:
%% \author[label1,label2]{<author name>}
%% \address[label1]{<address>}
%% \address[label2]{<address>}

\author{Dagnachew Jenber Negash}

\address{Addis Ababa Science and Technology University\\Addis Ababa, Ethiopia\\Email: djdm$\_$101979@yahoo.com}

\begin{abstract}
%% Text of abstract
In number theory, the Erdős–Straus conjecture states that for all integers $n\geq 2$, the rational number $4/n$ can be expressed as the sum of three positive unit fractions. Paul Erdős and Ernst G. Straus formulated the conjecture in $1948$. The restriction that the three unit fractions be positive is essential to the difficulty of the problem, for if negative values were allowed the problem could always be solved. This paper presents an explicit solutions to this conjecture for all $n\ge 2$ excepting some $n$ such that $n\equiv 1$(mod$8$).
\end{abstract}

\begin{keyword}
Elementary number theory\sep Recurrence relation.
%% keywords here, in the form: keyword \sep keyword

%% MSC codes here, in the form: \MSC code \sep code
%% or \MSC[2008] code \sep code (2000 is the default)

\end{keyword}
\end{frontmatter}

%%
%% Start line numbering here if you want
%%
%\linenumbers

%% main text
\section{Introduction}
\label{S:1}
\justify
In number theory, the Erdős–Straus conjecture states that for all integers $n\geq 2$, the rational number $4/n$ can be expressed as the sum of three positive unit fractions. Paul Erdős and Ernst G. Straus formulated the conjecture in $1948$\cite{01}. It is one of many conjectures by Erdős.
\justify
If $n$ is a composite number, $n = pq$, then an expansion for $4/n$ could be found from an expansion for $4/p$ or $4/q$. Therefore, if a counterexample to the Erdős–Straus conjecture exists, the smallest $n$ forming a counterexample would have to be a prime number, and it can be further restricted to one of six infinite arithmetic progressions modulo $840$\cite{02}. Computer searches have verified the truth of the conjecture up to $n\leq10^{17}$\cite{03}, but proving it for all $n$ remains an open problem.
\justify
The restriction that the three unit fractions be positive is essential to the difficulty of the problem, for if negative values were allowed the problem could always be solved.
\justify
More formally, the conjecture states that, for every integer $n\geq 2$, there exist positive integers $x$, $y$, and $z$ such that
\begin{equation*}
\frac{4}{n}=\frac{1}{x}+\frac{1}{y}+\frac{1}{z}
\end{equation*}
For instance, for $n = 5$, there are two solutions:
\begin{equation*}
\frac{4}{5}=\frac{1}{2}+\frac{1}{4}+\frac{1}{20}=\frac{1}{2}+\frac{1}{5}+\frac{1}{10}
\end{equation*}
Some researchers additionally require these integers to be distinct from each other, while others allow them to be equal. For $n\geq 3$, it does not matter whether they are required to be distinct: if there exists a solution with any three integers $x$, $y$, and $z$ then there exists a solution with distinct integers. For $n = 2$, however, the only solution is $4/2 = 1/2 + 1/2 + 1/1$, up to permutation of the summands. When $x$, $y$, and $z$ are distinct then these unit fractions form an Egyptian fraction representation of the number $4/n$.
\justify
The greedy algorithm for Egyptian fractions, first described in $1202$ by Fibonacci in his book Liber Abaci, finds an expansion in which each successive term is the largest unit fraction that is no larger than the remaining number to be represented. For fractions of the form $2/n$ or $3/n$, the greedy algorithm uses at most two or three terms respectively. More generally, it can be shown that a number of the form $3/n$ has a two-term expansion if and only if n has a factor congruent to $2$ modulo $3$, and requires three terms in any expansion otherwise. Thus, for the numerators $2$ and $3$, the question of how many terms are needed in an Egyptian fraction is completely settled, and fractions of the form $4/n$ are the first case in which the worst-case length of an expansion remains unknown. The greedy algorithm produces expansions of length two, three, or four depending on the value of $n$ modulo $4$; when $n$ is congruent to $1$ modulo $4$, the greedy algorithm produces four-term expansions. Therefore, the worst-case length of an Egyptian fraction of $4/n$ must be either three or four. The Erdős–Straus conjecture states that, in this case, as in the case for the numerator $3$, the maximum number of terms in an expansion is three\cite{05}.
\justify
Multiplying both sides of the equation $4/n = 1/x + 1/y + 1/z$ by $nxyz$ leads to an equivalent form $4xyz = n(xy + xz + yz)$ for the problem\cite{06}. As a polynomial equation with integer variables, this is an example of a Diophantine equation. The Hasse principle for Diophantine equations asserts that an integer solution of a Diophantine equation should be formed by combining solutions obtained modulo each possible prime number. On the face of it this principle makes little sense for the Erdős–Straus conjecture, as the equation $4xyz = n(xy + xz + yz)$ is easily solvable modulo any prime. Nevertheless, modular identities have proven a very important tool in the study of the conjecture.
\justify
For values of $n$ satisfying certain congruence relations, one can find an expansion for $4/n$ automatically as an instance of a polynomial identity. For instance, whenever $n \equiv 2 (mod 3)$, $4/n$ has the expansion 
\begin{equation*}
\frac{4}{n}=\frac{1}{n}+\frac{1}{(n+1)/3}+\frac{1}{n(n+1)/3}
\end{equation*}
\justify
Here each of the three denominators $n$, $(n + 1)/3$, and $n(n + 1)/3$ is a polynomial of $n$, and each is an integer whenever $n$ is $2$(mod 3). The greedy algorithm for Egyptian fractions finds a solution in three or fewer terms whenever $n$ is not $1$ or $17$(mod 24), and the $n \equiv$ 17(mod 24) case is covered by the $2 $(mod 3) relation, so the only values of $n$ for which these two methods do not find expansions in three or fewer terms are those congruent to $1 $(mod 24).
\justify
If it were possible to find solutions such as the ones above for enough different moduli, forming a complete covering system of congruences, the problem would be solved. However, as Mordell $(1967)$ showed, a polynomial identity that provides a solution for values of $n$ congruent to $r mod p$ can exist only when $r$ is not a quadratic residue modulo $p$. For instance, $2$ is a not a quadratic residue modulo $3$, so the existence of an identity for values of $n$ that are congruent to $2$ modulo $3$ does not contradict Mordell's result, but $1$ is a quadratic residue modulo $3$ so the result proves that there can be no similar identity for all values of n that are congruent to $1$ modulo $3$. As $1$ is a quadratic residue modulo $n(n > 1)$, there can be no complete covering system of modular identities for all $n$.
\justify
Polynomial identities listed by Mordell provide three-term Egyptian fractions for $4/n$ whenever $n$ is $2$ mod $3$ (above), $3$ mod $4$, $2$ or $3$ mod $5$, $3$, $5$, or $6$ mod $7$, or $5$ mod $8$ ($2$, $3$, $6$ and $7$ mod $8$ are already covered by earlier identities). These identities cover all the numbers that are not quadratic residues for those bases. However, for larger bases, not all nonresidues are known to be covered by relations of this type. From Mordell's identities one can conclude that there exists a solution for all n except possibly those that are $1$, $121$, $169$, $289$, $361$, or $529$ modulo $840$. $1009$ is the smallest prime number that is not covered by this system of congruences. By combining larger classes of modular identities, Webb and others showed that the fraction of $n$ in the interval $[1,N]$ that can be counterexamples to the conjecture tends to zero in the limit as $N$ goes to infinity\cite{07}.
\justify
Despite Mordell's result limiting the form these congruence identities can take, there is still some hope of using modular identities to prove the Erdős–Straus conjecture. No prime number can be a square, so by the Hasse–Minkowski theorem, whenever $p$ is prime, there exists a larger prime $q$ such that $p$ is not a quadratic residue modulo $q$. One possible approach to proving the conjecture would be to find for each prime $p$ a larger prime $q$ and a congruence solving the $4/n$ problem for $n \equiv p$ (mod q); if this could be done, no prime $p$ could be a counterexample to the conjecture and the conjecture would be true.
\justify
Various authors have performed brute-force searches for counterexamples to the conjecture; these searches can be greatly speed up by considering only prime numbers that are not covered by known congruence relations\cite{08}. Searches of this type have confirmed that the conjecture is true for all $n$ up to $10^{17}$\cite{03}.
\justify
A generalized version of the conjecture states that, for any positive $k$ there exists a number $N$ such that, for all $n\geq N$, there exists a solution in positive integers to $k/n = 1/x + 1/y + 1/z$. The version of this conjecture for $k = 5$ was made by Wacław Sierpiński, and the full conjecture is due to Andrzej Schinzel\cite{11}.
\justify
Even if the generalized conjecture is false for any fixed value of $k$, then the number of fractions $k/n$ with $n$ in the range from $1$ to $N$ that do not have three-term expansions must grow only sublinearly as a function of $N$\cite{07}. In particular, if the Erdős–Straus conjecture itself (the case $k = 4$) is false, then the number of counterexamples grows only sublinearly. Even more strongly, for any fixed $k$, only a sublinear number of values of $n$ need more than two terms in their Egyptian fraction expansions\cite{12}. The generalized version of the conjecture is equivalent to the statement that the number of unexpandable fractions is not just sublinear but bounded.
\justify
When $n$ is an odd number, by analogy to the problem of odd greedy expansions for Egyptian fractions, one may ask for solutions to $k/n = 1/x + 1/y + 1/z$ in which $x, y,$ and $z$ are distinct positive odd numbers. Solutions to this equation are known to always exist for the case in which $k = 3$\cite{13}. 
\section{Preliminary}
\label{S:2}
\begin{lemma}
For all $n\in \mathbb{Z+}$,
\begin{equation}
\{ 6n-1 \}_{n=1}^{\infty}\cup \{ 6n+1 \}_{n=1}^{\infty}= \{ 2n-1  \}_{n=3}^{\infty}
\end{equation}
\begin{eqnarray*}
\{ 6n+1  \}_{n=1}^{\infty}-\{8n+1\}_{n=1}^{\infty}
\end{eqnarray*}
\begin{eqnarray}
=\{ 4n-1 \}_{n=1}^{\infty}\cup\{ (4n-1)^2 \}_{n=1}^{\infty}\cup \{ 8n-3 \}_{n=1}^{\infty}\cup\{ (8n-3)^2 \}_{n=1}^{\infty}
\end{eqnarray}
\begin{equation*}
\{ n \}_{n=2}^{\infty} = \{ 2n\}_{n=1}^{\infty} \cup \{ 10n-1 \}_{n=1}^{\infty}\cup \{ 10n-3 \}_{n=1}^{\infty}\cup\{ 10n-5 \}_{n=1}^{\infty}
\end{equation*}
\begin{equation*}
\cup \{ 10n-7 \}_{n=1}^{\infty}\cup \{ 10n-9\}_{n=1}^{\infty}
\end{equation*}

\end{lemma}
\begin{lemma}
For all $k\in \mathbb{Z+}$, and $n=1,2,3,\cdots$
\begin{equation}
a_n=\frac{1}{3}\bigg[(3k+1)2^{2n+2}-1 \bigg];\text{ }a_0=4k+1
\end{equation}
\begin{equation}
b_n=\frac{1}{3}\bigg[(3k-1)2^{2n+1}-1 \bigg]; \text{ } b_0=2k-1
\end{equation}
\begin{equation*}
\frac{4}{2k}=\frac{1}{2k}+\frac{1}{2k}+\frac{1}{k}
\end{equation*}
\begin{equation*}
\frac{4}{3k}=\frac{1}{2k}+\frac{1}{2k}+\frac{1}{3k}
\end{equation*}
\end{lemma}
\section{Main Result}
\label{S:3}
\begin{theorem}For all non zero $c_1,c_2\in \mathbb{Z+}$
\begin{equation}
\frac{4}{4c_1c_2-c_1-c_2}=\frac{1}{c_1c_2}+\frac{1}{c_2(4c_1c_2-c_1-c_2)}+\frac{1}{c_1(4c_1c_2-c_1-c_2)}
\end{equation}
\begin{eqnarray*}
\frac{4}{(2c_1+1)(2c_2-1)(4(c_1-c_2)+1)}
=\frac{1}{c_1c_2}+\frac{1}{2c_1c_2(2c_1+1)(2c_2-1)}
\end{eqnarray*}
\begin{eqnarray}
+\frac{1}{(2c_1c_2)(2c_1+1)(2c_2-1)(4(c_1-c_2)+1)}
\end{eqnarray}
\begin{eqnarray*}
\frac{4}{(2c_1-1)(2c_2-1)(4c_1-1)(4c_2-1)}
=\frac{1}{c_1c_2(4c_1-1)(4c_2-1)}
\end{eqnarray*}
\begin{eqnarray}
+\frac{1}{2c_1c_2(2c_1-1)(2c_2-1)(4c_2-1)}+\frac{1}{2c_1c_2(2c_1-1)(2c_2-1)(4c_1-1)}
\end{eqnarray}
\begin{equation*}
\frac{4}{(4c_1+1)(4c_2-1)}=\frac{1}{(c_1+c_2)(4c_2-1)}
+\frac{1}{2(c_1+c_2)(4c_1+1)}
\end{equation*}

\begin{equation}
+\frac{1}{2(c_1+c_2)(4c_1+1)}
\end{equation}
\end{theorem}
\begin{proof}Let's see the proof of equation $(6)$ and we can follow the same process for equation $(5)$, $(7)$ and $(8)$.
\begin{equation*}
8c_1c_2=8c_1c_2-4c_1+4c_2-2+4c_1-4c_2+2
\end{equation*}
\begin{equation*}
=2[4c_1c_2-2c_1+2c_2-1]+(4c_1-4c_2+1)+1
\end{equation*}
\begin{equation*}
=2[(2c_1+1)(2c_2-1)]+(4c_1-4c_2+1)+1
\end{equation*}
After dividing both sides by $2(c_1c_2)(2c_1+1)(2c_2-1)(4c_1-4c_2+1)$, we will get the desired result and this completes the proof. 
\end{proof}
\begin{corollary}
\begin{equation}
\frac{4}{6k-1}=\frac{1}{2k}+\frac{1}{6k-1}+\frac{1}{2k(6k-1)}
\end{equation}
\end{corollary}
\begin{proof}
From Theorem $1$, equation $(5)$ for $c_1=1$ and $c_2=2k$.
\end{proof}

\begin{corollary}
\begin{eqnarray}
\frac{4}{(4k-1)}
=\frac{1}{k}+\frac{1}{2k(4k-1)}+\frac{1}{2k(4k-1)}
\end{eqnarray}
\begin{eqnarray}
\frac{4}{(4k-1)^2}
=\frac{1}{k(4k-1)}+\frac{1}{2k(4k-1)^2}+\frac{1}{2k(4k-1)^2}
\end{eqnarray}
\end{corollary}
\begin{proof}
From Theorem $1$, equation $(6)$ for $c_1=c_2$ and $c^2=k$.
\end{proof}
\begin{theorem}For all $k \in \mathbb{Z+}$
\begin{equation}
\frac{4}{8k-3}=\frac{1}{3k-1}+\frac{1}{2(3k-1)}+\frac{1}{2(3k-1)(8k-3)}
\end{equation}
\begin{equation}
\frac{4}{(8k-3)^2}=\frac{1}{(3k-1)(8k-3)}+\frac{1}{2(3k-1)(8k-3)}+\frac{1}{2(3k-1)(8k-3)^2}
\end{equation}

\end{theorem}

\begin{proof}
Follows from Lemma $2$, recurrence relation $b_n$ or $a_n$ for $n=1$.
\end{proof}
\subsection{\bf Solutions to Erdős–Straus Conjecture for some $n$ with unit digit $1$}
\label{Su:1}
\vspace{4mm}
Eventhough the following Corollaries can represent solutions for numbers otherthan numbers with unit digit $1$, it can be converted to the desired ones by letting $k=5m$ or $k=10m$ 

\begin{corollary}For all non zero positive integers $n_1,n_2\text{ and }k$, then
\begin{equation*}
\frac{4}{4k(10n_1)(10n_2-1)-10(n_1+n_2)+1}=
\end{equation*}
\begin{equation*}
\frac{1}{k(10n_1)(10n_2-9)}
+\frac{1}{k(10n_1)(4k(10n_1)(10n_2-1)-10(n_1+n_2)+1)}
\end{equation*}

\begin{equation}
+\frac{1}{k(10n_2-1)(4k(10n_1)(10n_2-1)-10(n_1+n_2)+1)}
\end{equation}
\end{corollary}
\begin{proof}
Followed from Theorem $1$, equation $(5)$ for $c_1=(10n_1)k$ and $c_2=(10n_2-1)k$
\end{proof}

\begin{corollary}For all non zero positive integers $n_1,n_2\text{ and }k$, then
\begin{equation*}
\frac{4}{4k(10n_1-9)(10n_2-2)-10(n_1+n_2)+11}=
\end{equation*}
\begin{equation*}
\frac{1}{k(10n_1-9)(10n_2-2)}
+\frac{1}{k(10n_1-9)(4k(10n_1-9)(10n_2-2)-10(n_1+n_2)+11)}
\end{equation*}

\begin{equation}
+\frac{1}{k(10n_2-2)(4k(10n_1-9)(10n_2-2)-10(n_1+n_2)+11)}
\end{equation}
\end{corollary}
\begin{proof}
Followed from Theorem $1$, equation $(5)$ for $c_1=(10n_1-9)k$ and $c_2=(10n_2-2)k$
\end{proof}
\begin{corollary}For all non zero positive integers $n_1,n_2\text{ and }k$, then
\begin{equation*}
\frac{4}{4k(10n_1-8)(10n_2-3)-10(n_1+n_2)+11}=
\end{equation*}
\begin{equation*}
\frac{1}{k(10n_1-8)(10n_2-3)}
+\frac{1}{k(10n_1-8)(4k(10n_1-8)(10n_2-3)-10(n_1+n_2)+11)}
\end{equation*}

\begin{equation}
+\frac{1}{k(10n_2-3)(4k(10n_1-8)(10n_2-3)-10(n_1+n_2)+11)}
\end{equation}
\end{corollary}
\begin{proof}
Followed from Theorem $1$, equation $(5)$ for $c_1=(10n_1-8)k$ and $c_2=(10n_2-3)k$
\end{proof}

\begin{corollary}For all non zero positive integers $n_1,n_2\text{ and }k$, then
\begin{equation*}
\frac{4}{4k(10n_1-7)(10n_2-4)-10(n_1+n_2)+11}=
\end{equation*}
\begin{equation*}
\frac{1}{k(10n_1-7)(10n_2-4)}
+\frac{1}{k(10n_1-7)(4k(10n_1-7)(10n_2-4)-10(n_1+n_2)+11)}
\end{equation*}

\begin{equation}
+\frac{1}{k(10n_2-4)(4k(10n_1-7)(10n_2-4)-10(n_1+n_2)+11)}
\end{equation}
\end{corollary}
\begin{proof}
Followed from Theorem $1$, equation $(5)$ for $c_1=(10n_1-7)k$ and $c_2=(10n_2-4)k$
\end{proof}

\begin{corollary}For all non zero positive integers $n_1,n_2\text{ and }k$, then
\begin{equation*}
\frac{4}{4k(10n_1-6)(10n_2-5)-10(n_1+n_2)+11}=
\end{equation*}
\begin{equation*}
\frac{1}{k(10n_1-6)(10n_2-5)}
+\frac{1}{k(10n_1-6)(4k(10n_1-6)(10n_2-5)-10(n_1+n_2)+11)}
\end{equation*}

\begin{equation}
+\frac{1}{k(10n_2-5)(4k(10n_1-6)(10n_2-5)-10(n_1+n_2)+11)}
\end{equation}
\end{corollary}
\begin{proof}
Followed from Theorem $1$, equation $(5)$ for $c_1=(10n_1-6)k$ and $c_2=(10n_2-5)k$
\end{proof}

\subsection{\bf Solutions to Erdős–Straus Conjecture for some $n$ with unit digit $3$}
\label{Su:2}
\vspace{4mm}
Eventhough the following Corollaries can represent solutions for numbers otherthan numbers with unit digit $3$, it can be converted to the desired ones by letting $k=5m$ or $k=10m$ 
\begin{corollary}For all non zero positive integers $n_1,n_2\text{ and }k$, then
\begin{equation*}
\frac{4}{4k(10n_1)(10n_2-3)-10(n_1+n_2)+3}=
\end{equation*}
\begin{equation*}
\frac{1}{k(10n_1)(10n_2-3)}
+\frac{1}{k(10n_1)(4k(10n_1)(10n_2-3)-10(n_1+n_2)+3)}
\end{equation*}

\begin{equation}
+\frac{1}{k(10n_2-3)(4k(10n_1)(10n_2-3)-10(n_1+n_2)+3)}
\end{equation}
\end{corollary}
\begin{proof}
Followed from Theorem $1$, equation $(5)$ for $c_1=(10n_1)k$ and $c_2=(10n_2-3)k$
\end{proof}
\begin{corollary}For all non zero positive integers $n_1,n_2\text{ and }k$, then
\begin{equation*}
\frac{4}{4k(10n_1-9)(10n_2-4)-10(n_1+n_2)+13}=
\end{equation*}
\begin{equation*}
\frac{1}{k(10n_1-9)(10n_2-4)}
+\frac{1}{k(10n_1-9)(4k(10n_1-9)(10n_2-4)-10(n_1+n_2)+13)}
\end{equation*}

\begin{equation}
+\frac{1}{k(10n_2-4)(4k(10n_1-9)(10n_2-4)-10(n_1+n_2)+13)}
\end{equation}
\end{corollary}
\begin{proof}
Followed from Theorem $1$, equation $(5)$ for $c_1=(10n_1-9)k$ and $c_2=(10n_2-4)k$
\end{proof}

\begin{corollary}For all non zero positive integers $n_1,n_2\text{ and }k$, then
\begin{equation*}
\frac{4}{4k(10n_1-8)(10n_2-5)-10(n_1+n_2)+13}=
\end{equation*}
\begin{equation*}
\frac{1}{k(10n_1-8)(10n_2-5)}
+\frac{1}{k(10n_1-8)(4k(10n_1-8)(10n_2-5)-10(n_1+n_2)+13)}
\end{equation*}

\begin{equation}
+\frac{1}{k(10n_2-5)(4k(10n_1-8)(10n_2-5)-10(n_1+n_2)+13)}
\end{equation}
\end{corollary}
\begin{proof}
Followed from Theorem $1$, equation $(5)$ for $c_1=(10n_1-8)k$ and $c_2=(10n_2-5)k$
\end{proof}

\begin{corollary}For all non zero positive integers $n_1,n_2\text{ and }k$, then
\begin{equation*}
\frac{4}{4k(10n_1-7)(10n_2-6)-10(n_1+n_2)+13}=
\end{equation*}
\begin{equation*}
\frac{1}{k(10n_1-7)(10n_2-6)}
+\frac{1}{k(10n_1-7)(4k(10n_1-7)(10n_2-6)-10(n_1+n_2)+13)}
\end{equation*}

\begin{equation}
+\frac{1}{k(10n_2-6)(4k(10n_1-7)(10n_2-6)-10(n_1+n_2)+13)}
\end{equation}
\end{corollary}
\begin{proof}
Followed from Theorem $1$, equation $(5)$ for $c_1=(10n_1-7)k$ and $c_2=(10n_2-6)k$
\end{proof}

\subsection{\bf Solutions to Erdős–Straus Conjecture for some $n$ with unit digit $5$}
\label{Su:3}
\vspace{4mm}
Eventhough the following Corollaries can represent solutions for numbers otherthan numbers with unit digit $5$, it can be converted to the desired ones by letting $k=5m$ or $k=10m$ 
\begin{corollary}For all non zero positive integers $n_1,n_2\text{ and }k$, then
\begin{equation*}
\frac{4}{4k(10n_1)(10n_2-5)-10(n_1+n_2)+5}=
\end{equation*}
\begin{equation*}
\frac{1}{k(10n_1)(10n_2-5)}
+\frac{1}{k(10n_1)(4k(10n_1)(10n_2-5)-10(n_1+n_2)+5)}
\end{equation*}

\begin{equation}
+\frac{1}{k(10n_2-5)(4k(10n_1)(10n_2-5)-10(n_1+n_2)+5)}
\end{equation}
\end{corollary}
\begin{proof}
Followed from Theorem $1$, equation $(5)$ for $c_1=(10n_1)k$ and $c_2=(10n_2-5)k$
\end{proof}
\begin{corollary}For all non zero positive integers $n_1,n_2\text{ and }k$, then
\begin{equation*}
\frac{4}{4k(10n_1-9)(10n_2-6)-10(n_1+n_2)+15}=
\end{equation*}
\begin{equation*}
\frac{1}{k(10n_1-9)(10n_2-6)}
+\frac{1}{k(10n_1-9)(4k(10n_1-9)(10n_2-6)-10(n_1+n_2)+15)}
\end{equation*}

\begin{equation}
+\frac{1}{k(10n_2-6)(4k(10n_1-9)(10n_2-6)-10(n_1+n_2)+15)}
\end{equation}
\end{corollary}
\begin{proof}
Followed from Theorem $1$, equation $(5)$ for $c_1=(10n_1-9)k$ and $c_2=(10n_2-6)k$
\end{proof}

\begin{corollary}For all non zero positive integers $n_1,n_2\text{ and }k$, then
\begin{equation*}
\frac{4}{4k(10n_1-8)(10n_2-7)-10(n_1+n_2)+15}=
\end{equation*}
\begin{equation*}
\frac{1}{k(10n_1-8)(10n_2-7)}
+\frac{1}{k(10n_1-8)(4k(10n_1-8)(10n_2-7)-10(n_1+n_2)+15)}
\end{equation*}

\begin{equation}
+\frac{1}{k(10n_2-7)(4k(10n_1-8)(10n_2-7)-10(n_1+n_2)+15)}
\end{equation}
\end{corollary}
\begin{proof}
Followed from Theorem $1$, equation $(5)$ for $c_1=(10n_1-8)k$ and $c_2=(10n_2-7)k$
\end{proof}
\subsection{\bf Solutions to Erdős–Straus Conjecture for some $n$ with unit digit $7$}
\label{Su:4}
\vspace{4mm}
Eventhough the following Corollaries can represent solutions for numbers otherthan numbers with unit digit $7$, it can be converted to the desired ones by letting $k=5m$ or $k=10m$ 
\begin{corollary}For all non zero positive integers $n_1,n_2\text{ and }k$, then
\begin{equation*}
\frac{4}{4k(10n_1)(10n_2-7)-10(n_1+n_2)+7}=
\end{equation*}
\begin{equation*}
\frac{1}{k(10n_1)(10n_2-7)}
+\frac{1}{k(10n_1)(4k(10n_1)(10n_2-7)-10(n_1+n_2)+7)}
\end{equation*}

\begin{equation}
+\frac{1}{k(10n_2-7)(4k(10n_1)(10n_2-7)-10(n_1+n_2)+7)}
\end{equation}
\end{corollary}
\begin{proof}
Followed from Theorem $1$, equation $(5)$ for $c_1=(10n_1)k$ and $c_2=(10n_2-7)k$
\end{proof}
\begin{corollary}For all non zero positive integers $n_1,n_2\text{ and }k$, then
\begin{equation*}
\frac{4}{4k(10n_1-9)(10n_2-8)-10(n_1+n_2)+17}=
\end{equation*}
\begin{equation*}
\frac{1}{k(10n_1-9)(10n_2-8)}
+\frac{1}{k(10n_1-9)(4k(10n_1-9)(10n_2-8)-10(n_1+n_2)+17)}
\end{equation*}

\begin{equation}
+\frac{1}{k(10n_2-8)(4k(10n_1-9)(10n_2-8)-10(n_1+n_2)+15)}
\end{equation}
\end{corollary}
\begin{proof}
Followed from Theorem $1$, equation $(5)$ for $c_1=(10n_1-9)k$ and $c_2=(10n_2-8)k$
\end{proof}
\subsection{\bf Solutions to Erdős–Straus Conjecture for some $n$ with unit digit $9$}
\label{Su:4}
\vspace{4mm}
Eventhough the following Corollaries can represent solutions for numbers otherthan numbers with unit digit $9$, it can be converted to the desired ones by letting $k=5m$ or $k=10m$ 
\begin{corollary}For all non zero positive integers $n_1,n_2\text{ and }k$, then
\begin{equation*}
\frac{4}{4k(10n_1)(10n_2-9)-10(n_1+n_2)+9}=
\end{equation*}
\begin{equation*}
\frac{1}{k(10n_1)(10n_2-9)}
+\frac{1}{k(10n_1)(4k(10n_1)(10n_2-9)-10(n_1+n_2)+9)}
\end{equation*}

\begin{equation}
+\frac{1}{k(10n_2-9)(4k(10n_1)(10n_2-9)-10(n_1+n_2)+9)}
\end{equation}
\end{corollary}
\begin{proof}
Followed from Theorem $1$, equation $(5)$ for $c_1=(10n_1)k$ and $c_2=(10n_2-9)k$
\end{proof}

%\begin{table}[h]
%\centering
%\begin{tabular}{l l l}
%\hline
%\textbf{Treatments} & \textbf{Response 1} & \textbf{Response 2}\\
%\hline
%Treatment 1 & 0.0003262 & 0.562 \\
%Treatment 2 & 0.0015681 & 0.910 \\
%Treatment 3 & 0.0009271 & 0.296 \\
%\hline
%\end{tabular}
%\caption{Table caption}
%\end{table}

%\begin{figure}[h]
%\centering\includegraphics[width=0.4\linewidth]{placeholder}
%\caption{Figure caption}
%\end{figure}

%% The Appendices part is started with the command \appendix;
%% appendix sections are then done as normal sections
%% \appendix

%% \section{}
%% \label{}

%% References
%%
%% Following citation commands can be used in the body text:
%% Usage of \cite is as follows:
%%   \cite{key}          ==>>  [#]
%%   \cite[chap. 2]{key} ==>>  [#, chap. 2]
%%   \citet{key}         ==>>  Author [#]

%% References with bibTeX database:
\newpage
\bibliographystyle{model1-num-names}
\bibliography{sample.bib}

%% Authors are advised to submit their bibtex database files. They are
%% requested to list a bibtex style file in the manuscript if they do
%% not want to use model1-num-names.bst.

%References without bibTeX database:

\end{document}